\documentclass[a4paper]{amsart}

\usepackage{color,cite}
\usepackage{amsfonts}
\usepackage{amssymb}
\usepackage{slashed}
\usepackage{tensor}

\usepackage{todonotes}
\usepackage[all]{xy}

\newcommand{\K}{\mathcal{K}}

\newcommand{\Z}{\mathbb{Z}}
\newcommand{\C}{\mathbb{C}}

\newcommand{\N}{\mathbb{N}}

\newcommand{\Cs}{$C^*$-}

\newcommand{\End}{\textnormal{End}}
\newcommand{\BndB}{\End^*_B}

\renewcommand{\hat}{\widehat}

\newcommand{\Pim}{\mathcal{O}_X}

\DeclareMathOperator{\tr}{tr}

\DeclareMathOperator{\Aut}{Aut}

\DeclareMathOperator{\Ext}{Ext}

\newcommand{\beq}{\begin{equation}}
\newcommand{\eeq}{\end{equation}}

\newcommand{\ot}{\otimes}
\newcommand{\hot}{\widehat \otimes}

\newcommand{\inn}[3]{\tensor[_{#1}]{\langle #2 \rangle}{_{#3}}}

\newcommand{\id}{\mathbf{1}}

 \newtheorem{thm}{Theorem}[section]
 
 \newtheorem{lem}[thm]{Lemma}
 \newtheorem{prop}[thm]{Proposition}
 \theoremstyle{definition}
 \newtheorem{defn}[thm]{Definition}
 \theoremstyle{remark}
 
 \newtheorem*{ex}{Example}

\begin{document}

%
%
%
%
%
%
%
%
%

\title[On Toeplitz extensions]
 {Toeplitz extensions in noncommutative \\ topology and mathematical physics}

 \author{Francesca Arici and Bram Mesland}

\address{Mathematical Institute\\
  Leiden University\\ P.O. Box 9512\\
  2300 RA Leiden\\ the Netherlands}
\email{f.arici@math.leidenuniv.nl}
\email{b.mesland@math.leidenuniv.nl}

\thanks{FA was partially funded by the Netherlands Organisation of Scientific Research (NWO) under the VENI grant 016.192.237.}


\date{\today}

\begin{abstract}
We review the theory of Toeplitz extensions and their role in operator $K$-theory, including Kasparov's bivariant $K$-theory. We then discuss the recent applications of Toeplitz algebras in the study of solid state systems, focusing in particular on the bulk-edge correspondence for topological insulators.
\end{abstract}

\maketitle

\section{Introduction}
Noncommutative topology is rooted in the equivalence of categories between locally compact topological spaces and commutative \Cs algebras. This duality allows for a transfer of ideas, constructions, and results between topology and operator algebras. This interplay has been fruitful for the advancement of both fields. Notable examples are the Connes--Skandalis foliation index theorem \cite{CoSk84}, the K-theory proof of the Atiyah--Singer index theorem \cite{AS1,At89}, and Cuntz's proof of Bott periodicity in $K$-theory \cite{Cu84}. Each of these demonstrates how techniques from operator algebras lead to new results in topology, or simplifies their proofs. In the other direction, Connes' development of \emph{noncommutative geometry} \cite{Co94} by using techniques from Riemannian geometry to study $C^{*}$-algebras, led to the discovery of cyclic homology \cite{Co83}, a homology theory for noncommutative algebras that generalises de Rham cohomology.

Noncommutative geometry and topology techniques have found ample applications in mathematical physics, ranging from Connes' reformulation of the standard model of particle physics \cite{Co96}, to quantum field theory \cite{CoMa}, and to solid-state physics. The noncommutative approach to the study of complex solid-state systems was initiated and developed in \cite{Be92,BevESB}, focusing on the quantum Hall effect and resulting in the computation of topological invariants via pairings between $K$-theory and cyclic homology. Noncommutative geometry techniques have proven to be a key tool in this field, and applications include the study of disordered systems, quasi-crystals and aperiodic solids \cite{Pro,ProSB}. The correct framework to describe such systems, as has been shown recently, is via $KK$-theory elements for certain observable \Cs algebras. 

This review is dedicated to a discussion of Toeplitz algebras and more generally $C^{*}$-extensions, and their role in noncommutative index theory. It is aimed at readers interested in the more recent applications of Toeplitz extensions and should serve as a brief overview and introduction to the subject. We shall provide an exposition of operator algebra techniques recently used in mathematical physics, in particular in the study of solid state systems. 

The paper is structured as follows. In Section \ref{s:toe} we review the construction of the classical one-dimensional Toeplitz algebra as the universal \Cs algebra generated by a single isometry, and we recall its role in the Noether--Gohberg--Krein index theorem, which relates the index of Toeplitz operators to the winding number of their symbol. We conclude the section by discussing how the construction can be extended to higher dimensions. In Section \ref{s:toeKK} we take a deep dive into the world of noncommutative topology and discuss the role of Toeplitz extensions in operator $K$-theory, namely in Cuntz's proof of Bott periodicity and in the development of Kasparov's bivariant $K$-theory. This rather technical section allows us to introduce the tools that are needed in the noncommutative approach to solid state physics. 
In Section \ref{sec:ToePim}, we describe two constructions of universal \Cs algebras that will later play a crucial role in the study of solid state systems, namely crossed products by the integers, Cuntz--Pimsner algebras, and their Toeplitz algebras. Finally, Section \ref{sec:SolidState} is devoted to describing how Toeplitz extensions and the associated maps in $K$-theory provide the natural framework for implementing the bulk-edge correspondence from solid state physics.

\section{Toeplitz algebras of operators}
\label{s:toe}
\subsection{Shifts, winding numbers, and the Noether--Gohberg--Krein index theorem}
\label{ss:shifts}
In view of the Gelfand--Naimark theorem \cite{GN43}, every abstract \Cs algebra, commutative or not, admits a faithful representation as a subalgebra of the algebra $B(H)$ of bounded operators on some Hilbert space $H$. In this section, we will start by constructing two concrete examples of \Cs algebras of operators. As mentioned in the Introduction, we are interested in how the commutative algebra of functions on the circle and the noncommutative algebra generated by a single isometry fit together in a short exact sequence. This extension will later serve as our prototypical example illustrating the use of \Cs algebraic techniques in solid state physics. 

Let $ S^1 := \left \lbrace z \in \C \ \vert \ \overline{z}z = 1 \right \rbrace$ denote the unit circle in the complex plane. The corresponding \Cs algebra, $C(S^1)$, is 
the closure in the supremum norm of the algebra of \emph{Laurent polynomials}
\[\mathcal{O}(S^1) = \frac{\mathbb{C}[z, \overline{z}]}{\langle \overline{z}z = 1 \rangle}.\]
The algebra  $C(S^1)$ admits a convenient representation on the Hilbert space $L^2(S^1)$ of square-integrable functions on $S^1$. This Hilbert space is isomorphic to the Hilbert space of sequences $\ell^2(\Z)$, and the isomorphism is implemented by the discrete Fourier transform
\begin{equation}
\label{eq:Fourier}
\mathcal{F}:\ell^2(\Z) \to L^2(S^1), \qquad (\mathcal{F}\phi)(z)=(2\pi)^{-\frac{1}{2}} \sum_{n \in \Z}\phi_n e^{-i {n\cdot z}}.
\end{equation}
Under this isomorphism, the operator of multiplication by $z$ is mapped to the bilateral shift operator $U$, defined on the standard basis $\{e_{n}\}_{n\in\mathbb{Z}}$ of $\ell^{2}(\mathbb{Z})$ via 
\begin{equation}
\label{eq:rightUshift} 
U(e_n)=(e_{n+1}), \quad U^*(e_n)=e_{n-1}.\end{equation}
It is easy to see that $U$ is a \emph{unitary} operator, i.e. $U^*U=\id=UU^*$. The algebra $C(S^1)$ is then isomorphic to the smallest \Cs subalgebra of $B(\ell^2(\Z))$ that contains $U$.

In order to define the second \Cs algebra we are interested in, which is  genuinely non-commutative, we shall consider the Hardy space $H^2(S^1)$. This is defined as the subset of $L^2(S^1)$ consisting of holomorphic $L^2$-functions. The discrete Fourier transform allows us to identify the Hardy space with the sequence space $\ell^2(\N)$. We will denote by $\mathrm{p}$ the orthogonal projection from $\ell^2(\Z)$ to $\ell^2(\N)$, and by $\mathrm{P}$ that from $L^2(S^1)$ onto $H^2(S^1)$ (obtained by conjugating $\mathrm{p}$ with the Fourier transform).

Multiplication by $z$ on the Hardy space corresponds to a shift operator on $\ell^2(\N)$, called the \emph{unilateral shift}, expressed on the standard basis $\{f_{n}\}_{n\in\mathbb{N}}$ of $\ell^{2}(\mathbb{N})$ via:
\[T(f_n)=(f_{n+1}).\]
Its adjoint is not invertible, as
\[T^*(f_n)=\begin{cases}
f_{n-1} & n \geq 1 \\
0 & n=0 
\end{cases}.\]
This motivates the following:
\begin{defn}
\label{def:ToeAlg}
The \emph{Toeplitz algebra} $\mathcal{T}$ is the smallest \Cs subalgebra of $B(\ell^2(\N))$ that contains $T$. 
\end{defn}
It is easy to see that the Toeplitz algebra $\mathcal{T}$ is not commutative, as 
\begin{equation}
\label{eq:Toe_commrel}
T^*T=\id, \qquad TT^*= 1-\mathrm{p}_{\ker(T^*)}.
\end{equation}
In particular, it follows from \eqref{eq:Toe_commrel} that elements of $\mathcal{T}$ commute up to compact operators, and in particular the generator $T$ is unitary module compact operators. In other words, the Toeplitz algebra can be viewed as the \Cs algebra extension of continuous functions on the circle by the compact operators:
\begin{equation}
\label{eq:Toe_ext}
\xymatrix{0 \ar[r] & \mathcal{K}(\ell^2(\N)) \ar[r] & \mathcal{T} \ar[r]^{\pi} & C(S^1) \ar[r] & 0.}
\end{equation}
The extension \eqref{eq:Toe_ext} admits a completely positive and completely contractive splitting given by the Hardy projection. Indeed, for every $f \in C(S^1)$, the assignment 
\begin{equation}
\label{eq:Toe_op}
T_f(g) = \mathrm{P}(fg), \quad g \in H^2(S^1)
\end{equation}
defines a bounded operator on the Hardy space $H^2(S^1)$, where, under Fourier transform, $T_{z}$ corresponds to the unilateral shift. As the function $z$ generates $C(S^1)$ as a \Cs algebra, every such $T_f$ is an element of $\mathcal{T}$.

The following result implies that the Toeplitz algebra is the universal \Cs algebra generated by an element $T$ satisfying $T^*T=1$:

\begin{thm}[Coburn, \cite{Co67}]
Suppose $v$ is an isometry in a unital \Cs algebra $A$. Let $T=T_{z} \in \mathcal{T}$. Then there exists a unique unital $*$-homomorphism $\phi: \mathcal{T} \to A$ such that $\phi(T)=v$. Moreover, if $vv^* \neq 1$, then the map $\phi$ is isometric. 
\end{thm}

\subsubsection{The Noether--Gohberg--Krein index theorem}

Recall that an operator $F\in B(H)$ is a \emph{Fredholm operator} if both $\ker F$ and $\ker F^{*}$ are finite-dimensional. The \emph{Fredholm index} of such an operator
is the integer \[
\mathrm{Ind}(F)= \dim \ker F-\dim\ker F^{*}\in\mathbb{Z}.\] One of the key properties of the Fredholm index is that it is constant along continuous paths of Fredholm operators.
As such it is a homotopy invariant.

The completely positive linear splitting $f\mapsto T_{f}$ allows one to give a precise characterization of which Toeplitz operators $T_{f}$ are Fredholm. Moreover,
the index of a Fredholm Toeplitz operator $T_{f}$ can be described entirely in terms of a familiar homotopy invariant of the complex function $f$. This is the content of the Toeplitz
index theorem, due to F. Noether and later reproved independently by Gohberg and Krein. It was one of the first results linking index theory to topology and should be viewed
as an ancestor to the celebrated Atiyah-Singer index theorem.
\begin{thm}[{Noether \cite{Noe20}, Gohberg--Krein \cite{GoKr}}] For $f:S^{1}\to \mathbb{C}^{\times}$ the operator $T_{f}:H^{2}(S^1) \to H^{2}(S^1)$ is Fredholm and
\[\mathrm{Ind}\left( T_{f} \right)=-w(f),\]
with $w(f)$ the winding number of $f$. If $f$ is a $C^{1}$-function, then the winding number can be computed as \[w(f)=\int_{S^{1}}\frac{f'(z)}{f(z)}\mathrm{d}z.\]
\end{thm}
The latter, explicit expression for the winding number and hence the Toeplitz index should be viewed as a result of \emph{differential} topology: By choosing a nice representative in the homotopy class of the function $f$, the differential calculus can be employed to compute a topological invariant. We will see an application of this computation in Section~\ref{sec:SolidState}.
\subsection{Generalisation: higher Toeplitz algebras}
\subsubsection{Toeplitz operators on strongly pseudo-convex domains}
The definition of Toe\-plitz operators on the circle in terms of the Hardy space lends itself to generalisations to higher dimensions. The crucial observation here is that the Hardy space $H^2(S^1)$ can be defined as the closure of the space of boundary values of holomorphic functions on the unit disk that admit a continuous extension to the closed unit disk.

\begin{defn}[{\!\cite[Definition 1.2.18]{Up96}}]
Let $\Omega $ be a smooth-domain in $\C^n$ with defining function $\rho \in C^{\infty}(\C^n)$:
\[ \Omega = \lbrace z \in \C^n \ : \rho(z)<0 \rbrace \]
and boundary $\partial\Omega=\lbrace z \in \C^n \ : \rho(z)=0 \rbrace$.

For every $z \in \partial \Omega$, the Levi form $\langle \, , \, \rangle_z$ is defined as \[ \langle u,v \rangle_z := \sum_{1 \geq i,j \geq n} \frac{\partial^2 \rho}{\partial z_i \partial \overline{z}_j} (z) u_j \overline{v}_j, \qquad u, v \in \C^n.\]

Then $\Omega$ is called a \emph{strongly pseudo-convex domain} if the Levi form is positive semi-definite on the complex tangent space at every point $z \in \partial \Omega$, i.e., $u \in T_z(\partial \Omega ), u \neq 0$ implies $\langle u,u \rangle_z >0$,
\end{defn}

Open balls in $\C^n$ are examples of strongly pseudo-convex domains. However, the product of two open balls is not strongly pseudo-convex, showing the notion is somewhat subtle. 

Given a strongly pseudo-convex domain $\Omega \subseteq \C^n$ with smooth boundary, we denote by $L^2(\partial \Omega)$ the Hilbert space of square integrable functions on the boundary $\partial \Omega$.  The Hardy space $H^2(\partial \Omega)$ is defined as the Hilbert space closure in $L^2(\partial \Omega)$ of boundary values of homolomorphic functions on $\Omega$ that admit a continuous extensions to the boundary $\partial \Omega$ (cf. \cite[Definition 2.3]{Up96}). The orthogonal projection
\[ \mathrm{P}_{CS}: L^2(\partial \Omega) \to H^2(\partial \Omega),\]  called the Cauchy--Szeg\"o projection, is used to define Toeplitz operators, in analogy with \eqref{eq:Toe_op}. Indeed, let $f$ be a continuous function on  $\partial \Omega$, the Toeplitz operator with symbol $f$ is defined as 
\[ T_f(g) = \mathrm{P}_{CS}(fg), \]
for all $g \in H^2(\partial \Omega)$. 

For any two $f, f^\prime \in C(\partial \Omega)$, the product of Toeplitz operators $T_f\circ T_{f^\prime}$ is equal to $T_{f f^\prime}$ modulo compact operators. Moreover, for any $f \in C(\partial \Omega)$, the operator $T_f$ is compact if and only if $f$ is identically zero. These two facts combined lead to the following:
\begin{thm}
Let $\Omega$ be a strictly pseudo-convex domain. There is an extension of \Cs algebras:
\begin{equation*}
\label{eq:Toe_ext_pseudo}
\xymatrix{0 \ar[r] & \mathcal{K}(H^2(\partial \Omega)) \ar[r] & \mathcal{T}(\partial \Omega) \ar[r]^{} & C(\partial \Omega) \ar[r] & 0.}
\end{equation*}
The extension admits a completely positive and completely contractive linear splitting given by the Cauchy--Szeg\"o projection.
\end{thm}
Applied to the unit ball in $\mathbb{C}^{n}$ this construction yields the Toeplitz extensions for odd-dimensional spheres as a special case:
\begin{equation*}
\label{eq:Toe_ext_spheres}
\xymatrix{0 \ar[r] & \mathcal{K}(H^2(S^{2d-1})) \ar[r] & \mathcal{T}(S^{2d-1}) \ar[r]^{} & C(S^{2d-1}) \ar[r] & 0,}
\end{equation*}
which clearly recover \eqref{eq:Toe_ext} for $d=1$.

The Toeplitz algebra $\mathcal{T}(S^{2d-1})$ admits an equivalent description in terms of so-called $d$-shifts, as described in \cite[Theorem 5.7]{Arv98}. For an overview of the interplay of Toeplitz $C^{*}$-algebras and index theory, as well as their role in the computation of noncommutative invariants, we refer the reader to the excellent survey \cite{Le}.

\section{Toeplitz algebras in operator $K$-theory and bivariant $K$-theory}
\label{s:toeKK}

Operator $K$-theory is a functor, associating to a $C^{*}$-algebra $A$ two  Abelian groups $K_{*}(A)$, $*=0,1$.
Functoriality means that for a $*$-homomorphism $\varphi:A\to B$ between $C^*$-algebras $A$ and $B$, there is an induced homomorphism of Abelian groups
\[\varphi_{*}:K_{*}(A)\to K_{*}(B).\]

The key properties of the operator $K$-theory functor are that it is \emph{homotopy invariant}, \emph{half-exact} and \emph{Morita invariant}.
We now define each of these properties more precisely.

Homotopy invariance is the property that if $\varphi$ and $\psi$ are connected by a continuous path of $*$-homomorphisms, then the induced maps on $K$-theory coincide, that is $\varphi_{*}=\psi_{*}$. 

Half-exactness is the property that for any extension of $C^{*}$-algebras
\begin{equation}
\label{extension}
\xymatrix{0 \ar[r] & I \ar[r]^{i} & E \ar[r]^{p} & A\ar[r] & 0,}
\end{equation}
the corresponding sequence of groups
\[ \xymatrix{K_{*}(I)\ar[r]^{i_*} & K_{*}(E) \ar[r]^{p_*} & K_{*}(A),}\]
is exact at $K_{*}(E)$. 

Lastly, Morita invariance entails that for any rank-one
projection $p\in\mathcal{K}=\mathcal{K}(\ell^{2}(\mathbb{N}))$, the $*$-homomorphism
\[A\to \mathcal{K}\otimes A,\quad a\mapsto p\otimes a,\]
induces an isomorphism in $K$-theory.

Recall that the \emph{suspension} $SA$ of a $C^{*}$-algebra $A$ is defined to be \[SA:=C_0(0,1)\otimes A\simeq C_{0}((0,1),A),\]
which is a $C^{*}$-algebra in the sup-norm, and pointwise product and involution inherited from $A$. 

The operation
$A\to SA$ is functorial for $*$-homorphisms, and it is customary to define the \emph{higher} $K$-groups
as $K_{n}(A):=K_{0}(S^{n}A)$. Via a general construction in topology, it follows that the extension \eqref{extension}
induces a long exact sequence
\begin{equation}
\label{exseq}
\cdots \to K_{n+1}(A)\to K_{n}(I)\to K_{n}(E)\to K_{n}(A)\to K_{n-1}(I)\to \cdots,
\end{equation}
of Abelian groups.

The boundary maps in such exact sequences are often related to index theory. For instance, for
the Toeplitz extension \eqref{eq:Toe_ext}, the boundary map 
\begin{equation}
\label{eq:ToeBndry}
\partial:K_{1}(C(S^{1}))\to K_{0}(\mathcal{K}(\ell^{2}(\mathbb{N}))\simeq \mathbb{Z},
\end{equation}
maps the class of a nonzero function $f\in C(S^{1})$ to the index of the Toeplitz operator $T_{f}$. 

One of the key features of operator $K$-theory is \emph{Bott periodicity}. It states that for any $C^{*}$-algebra $A$
there are natural isomorphisms between its $K$-theory and the $K$-theory of its double suspension $S^{2}A$. 
 It turns out that the three properties of homotopy invariance, half-exactness and Morita invariance suffice
to deduce the existence of natural \emph{Bott periodicity} isomorphisms $K_{*}(A)\simeq K_{*}(S^{2}A)$. As a consequence, there are only 
two $K$-functors, $K_0$ and $K_1$, and the exact sequence \eqref{exseq} reduces the cyclic six-term exact sequence 
\[
\label{eq:6es}
\xymatrix{
& K_0(I) \ar[r]^{i_*} & K_0(E) \ar[r]^{p_*} & K_0(A) \ar[d] \\
& K_1(A) \ar[u] &K_1(E)  \ar[l]^{p_*} & K_1(I)\ar[l]^{i_{*}} 
.}
\]
\subsection{Cuntz's proof of Bott periodicity}
Apart from the invariance properties of the $K$-functor, Cuntz's proof of Bott periodicity (cf. \cite{Cu84}) exploits essential properties of the Toeplitz extension \eqref{eq:Toe_ext}. By composing the projection homomorphism $\pi: \mathcal{T} \to C(S^1)$ with the evaluation map $\textnormal{ev}_{1}:C(S^{1})\to\mathbb{C}$, given by  $\mathrm{ev}_1(f) = f(1)$,
we obtain a character of $\mathcal{T}$:
\begin{equation}
\label{eq:chi}
\chi:=\textnormal{ev}_{1}\circ\pi:\mathcal{T}\to \mathbb{C}.
\end{equation} 
The unital embedding $\iota:\mathbb{C}\to\mathcal{T}$ splits the homomorphism $\chi$ in the sense that $\chi\circ\iota=\textnormal{id}_{\mathbb{C}}$.
It is a non-trivial fact these $*$-homomorphisms are mutually inverse in $K$-theory, in a strong sense made precise below. 

To state the result, which lies
at the heart of the proof of the Bott periodicity theorem, we shall recall the construction of the \emph{spatial} or \emph{minimal} tensor product $A_{1}\overline{\otimes}A_{2}$ of $C^{*}$-algebras $A_{i}, i=1,2$. 
Choose faithful representations $\pi_{i}:A_{i}\to B(\mathcal{H}_{i})$ and let $\mathcal{H}_{1}\otimes \mathcal{H}_2$ be the completed tensor product of Hilbert spaces. One defines $A\overline{\otimes}B$ to be the completion of the algebraic tensor product $A\otimes B$ 
in the norm inherited from the representation
\[\pi_{1}\otimes \pi_{2}:A_{1}\otimes A_{2}\to B(\mathcal{H}_{1}\otimes\mathcal{H}_{2}).\]

\begin{prop}[\!\!{\cite[Proposition 4.3]{Cu84}}]
\label{prop: toeplitzisawesome}
Let $A$ be a $C^{*}$-algebra. The map $\chi\otimes 1:\mathcal{T}\overline{\otimes}A\to A$ induces an isomorphism $\chi_{*}\otimes 1:K_{0}(\mathcal{T}\overline{\otimes}A)\xrightarrow{\sim} K_{0}(A)$.
\end{prop}
Tensor products of $C^{*}$-algebras are not unique, and the spatial tensor product is the completion in the minimal $C^{*}$-norm on the algebraic tensor product $A\otimes B$. There is also a maximal $C^{*}$-norm on $A\otimes B$, which involves taking the supremum over all representations. A $C^{*}$-algebra $N$ is \emph{nuclear}, if for any other $C^{*}$-algebra $A$, the minimal and maximal $C^{*}$-tensor norms on $N\otimes A$ coincide. For our purposes it suffices to know that all commutative $C^{*}$-algebras are nuclear. Given an extension of $C^{*}$-algebras
\begin{equation}
\label{absext}
\xymatrix{0 \ar[r] & I \ar[r] & E \ar[r] & B \ar[r] & 0},
\end{equation}
the sequence of tensor products
\begin{equation}
\label{tensabsext}
\xymatrix{0 \ar[r] & I\overline{\otimes} A \ar[r] & E\overline{\otimes} A \ar[r]  & B\overline{\otimes} A \ar[r] & 0},
\end{equation}
may fail to be exact in the middle. However, nuclearity of the $C^{*}$-algebra $B$ guarantees exactness.
\begin{lem}[{cf. \cite[Corollary 3.7.4]{BrOz08}}]\label{exactextension}
Let $A$ be a $C^{*}$-algebra and consider an extension \eqref{absext}. If the $C^{*}$-algebra $B$ is nuclear, then the sequence \eqref{tensabsext} is exact.
\end{lem}

We can now exploit Proposition \ref{prop: toeplitzisawesome}, Lemma \ref{exactextension}, and the exactness properties of the $K$-functor to deduce Bott periodicity.
\begin{thm}
\label{thm:BottPer}
For any $C^{*}$-algebra $A$ there are natural isomorphisms $K_{n}(A)\simeq K_{n+2}(A)$. 
\end{thm} 
\begin{proof}
Consider the character $\chi$ defined in \eqref{eq:chi} and let $\mathcal{T}_{0}:=\ker \chi$, so that we have an extension
\[\xymatrix{0 \ar[r] & \mathcal{T}_0 \ar[r] & \mathcal{T} \ar[r] & \mathbb{C} \ar[r] & 0}.\]
As $\mathbb{C}$ is nuclear, this extension has the property that the induced sequence
\[\xymatrix{0 \ar[r] & \mathcal{T}_0\overline{\otimes} A \ar[r] & \mathcal{T}\overline{\otimes} A \ar[r] & A \ar[r] & 0},\]
is exact for any $C^{*}$-algebra $A$ as well, by Lemma \ref{exactextension}.

The long exact sequence \eqref{exseq}, together with the fact that $S(A\overline{\otimes}B)\simeq A\overline{\otimes}SB$ and Proposition \ref{prop: toeplitzisawesome}, imply that $\chi_{*}:K_{n}(\mathcal{T}\overline{\otimes} A)\to K_{n}(A)$ is an isomorphism for all $n$. Consequently $K_{n}(\mathcal{T}_{0}\overline{\otimes}A)=0$ for all $n$. Now observe that, after identifying $\ker \textnormal{ev}_{1}$ with $C_{0}(0,1)$, we can construct a second extension
\[\xymatrix{0 \ar[r] & \mathcal{K}\ar[r] & \mathcal{T}_{0}\ar[r] & C_{0}(0,1) \ar[r] & 0}.\]
 As $C_0(0,1)$ is nuclear, this extension, too, has the property that
\[\xymatrix{0 \ar[r] & \mathcal{K}\overline{\otimes} A \ar[r] & \mathcal{T}_{0}\overline{\otimes} A \ar[r] & C_{0}(0,1)\overline{\otimes}A \ar[r] & 0}\]
is exact for any $C^{*}$-algebra $A$, by Lemma \ref{exactextension}. Since $C_{0}(0,1)\overline{\otimes}A\simeq SA$, the long exact sequence \eqref{exseq} gives an isomorphism
\[K_{n+1}(C(0,1)\overline{\otimes}A)\xrightarrow{\sim} K_n(\mathcal{K}\overline{\otimes}A).\]
Now we use the Morita invariance isomorphism $K_{n}(\mathcal{K}\overline{\otimes}A)\simeq K_{n}(A)$ and the fact that $C(0,1)\overline{\otimes}A\simeq SA$ to deduce that
\[K_{n+2}(A)\simeq K_{n+1}(C(0,1)\overline{\otimes}A)\xrightarrow{\sim} K_n(\mathcal{K}\overline{\otimes}A)\simeq K_{0}(A),\]
which yields the Bott periodicity isomorphism.
\end{proof}
We remark that, in fact, the theorem holds if we replace $K$ by any functor that is homotopy invariant, half-exact and Morita invariant.
\subsection{Toeplitz extensions and bivariant $K$-theory}
As we have seen so far in the Toeplitz index and Bott periodicity theorems, extensions of $C^{*}$-algebras play a crucial role in $K$-theory and henceforth in index theory. An extension of a C*-algebra $A$ by $B$ should be viewed as a new $C^{*}$-algebra, built by "gluing together" $A$ and $B$ in a possibly topologically nontrivial way.

In \cite{BDF77}, Brown, Douglas, and Fillmore initiated the study of extensions by considering exact sequences of the form
\[\xymatrix{0 \ar[r] & \mathcal{K}(H) \ar[r] & E \ar[r]^{} & C(M) \ar[r] & 0,}\]
for some Hilbert space $H$ and some compact Hausdorff topological space $M$. They proved that such extensions form an Abelian group by defining addition via an appropriate version of the Baer sum. They also showed that their Abelian group is dual to $K$-theory in a precise sense governed by Fredholm index theory.

 Kasparov generalized this construction to extensions
\[\xymatrix{0 \ar[r] & \mathcal{K}(X) \ar[r] & E \ar[r]^{} & A\ar[r] & 0,}\]
where $A$ is a separable $C^{*}$-algebra and $X$ a countably generated Hilbert $C^{*}$-module over a second, $\sigma$-unital $C^{*}$-algebra $B$. A technical assumption on such extensions 
is that they admit a completely positive and completely contractive linear splitting $\ell:A\to E$ such that $\ell\circ\pi=\textnormal{id}_{A}$.  This assumption is automatically satisfied when the quotient algebra in the extension is nuclear. Commutative $C^*$-algebras are nuclear, and thus  the Toeplitz extensions discussed previously satisfy this assumption. The isomorphism classes of such extensions form an Abelian group $\Ext^{1}(A,B)$ which is isomorphic to the Kasparov group $KK_{1}(A,B)$. This section is devoted to making this statement more precise. An excellent reference for this discussion is \cite[Chapter 3]{JenTho}.

\subsubsection{Hilbert modules and \Cs correspondences}
\label{sss:hilb}
Before we proceed, we need to recall some results from the theory of Hilbert \Cs modules. For more details on the latter, we refer the interested reader to the monograph \cite{La95}.
\begin{defn}
 A \emph{pre-Hilbert module} over $B$ is a right $B$-module $X$ with a $B$-valued Hermitian product, i.e. a map $\inn{}{\cdot,\cdot}{B}: X\times X \rightarrow B$ satisfying
\[\begin{aligned}
&  \inn{}{\xi,\eta}{B}&=& \langle \eta,\xi \rangle_{B}^*, \quad &  \inn{}{\xi, \eta b}{B} &= \inn{}{\xi,\eta}{B} b , \\&  \inn{}{\xi,\xi}{B} &\geq & \,0,\quad  
&  \inn{}{\xi, \xi}{B} &= 0 \Leftrightarrow \xi =0,
 \end{aligned}\]
for all $\xi, \eta \in X$ and for all $b \in B$. 
\end{defn}
For a pre-Hilbert module $X$, one can define a scalar valued norm $\| \cdot \|$ using the \Cs norm on $B$:
\begin{equation}
\label{Hilbnorm}
 \| \xi \|^2 = \| \inn{}{\xi,\xi}{B} \|_B.
\end{equation}

\begin{defn}
A \emph{Hilbert \Cs module} is a pre-Hilbert module that is complete in the norm \eqref{Hilbnorm}.
\end{defn}
If one defines $\inn{}{X,X}{}$ to be the linear span of elements of the form $\inn{}{\xi, \eta}{}$ for $\xi, \eta \in X$, then its closure its a two-sided ideal in $B$. We say that the Hilbert module $X$ is \emph{full} whenever $\inn{}{X,X}{}$ is dense in $B$.

Let now $X, Y$ be two Hilbert \Cs modules over the same \Cs algebra $B$.
\begin{defn}
 A map $T: X \rightarrow Y$ is said to be an \emph{adjointable operator} if there exists another map $T^*: Y\rightarrow X$ with the property that 
$$
 \inn{}{T\xi, \eta}{} = \inn{}{\xi, T^*\eta}{} \qquad \textup{for all} \quad \xi \in X, \eta \in Y\; .
$$
\end{defn}

Every adjointable operator is automatically right $B$-linear and bounded. However, the converse is in general not true: a bounded linear map between Hilbert modules need not be adjointable. 

We denote the collection of adjointable operators from $X$ to $Y$ by $\End^{*}_{B}(X,Y)$. When $X=Y$, the adjointable operators form a \Cs algebra in the operator norm, that is denoted by $\End^{*}_{B}(X)$.

Inside the adjointable operators one can single out a particular subspace, which is analogous to that of finite-rank operators on a Hilbert space.  More precisely, for every $\xi \in Y, \eta \in X$ one defines the operator $\theta_{\xi, \eta} : X \rightarrow Y$ as
\begin{equation}
\theta_{\xi, \eta} (\zeta) =  \xi  \inn{}{\eta, \zeta}{}, \qquad \forall \zeta \in X
\end{equation}
This is an adjointable operator, with adjoint $\theta_{\xi, \eta}^* : Y \rightarrow X$ given by $\theta_{\eta, \xi}$.

We denote by $\mathcal{K}_B(X,Y)$ the closure of the linear span of 
\begin{equation} 
 \lbrace \theta_{\xi, \eta} \, \vert \, \xi, \eta \in X \rbrace \subseteq \End^{*}_{B}(X,Y),
\end{equation}
and we refer to it as the space of \emph{compact adjointable operators}.

In particular $\mathcal{K}_B(X):= \mathcal{K}_B(X,X) \subseteq \End^{*}_{B}(X)$ is a closed two-sided ideal in the $C^{*}$-algebra $\End^{*}_{B}(X)$, hence a \Cs subalgebra, whose elements are referred to as \emph{compact endomorphisms}. Elements of $\mathcal{K}_B(X)$ and of $\End^*_B(X)$ act on $X$ from the left, motivating the following:
\begin{defn}
\label{d:c*corr}
A \emph{\Cs correspondence $(X,\phi)$ from $A$ to $B$},  is a right Hilbert $B$-module $X$ endowed with a $*$-homomorphism $\phi: A\to \End^{*}_{B}(X)$.  If $\phi:A\to\mathcal{K}_B(X)$ we refer to $(X,\phi)$ as a \emph{compact} $C^{*}$-correspondence and in case $A=B$ we refer to $(X, \phi)$ as a \Cs correspondence \emph{over $B$}. 
\end{defn}
\noindent When no confusion arises, we will omit the map $\phi$ and simply write $X$.

Two \Cs correspondences $X_{\phi}$ and $Y_{\psi}$ over the same algebra $B$ are called \emph{isomorphic} if and only if there exists a unitary $U \in \End^{*}_{B}(X,Y)$ intertwining $\phi$ and $\psi$.

Given an $(A,B)$-correspondence $X_\phi$ and a $(B,C)$-correspondence $Y_{\psi}$, one can construct an $(A,C)$-correspondence, named the \emph{interior tensor product} of $X_{\phi}$ and $Y_{\psi}$. 

As a first step, one constructs the \emph{balanced tensor product} $X \otimes_{B} Y$ which is a quotient of the algebraic tensor product $X \otimes_{\mathrm{alg}} Y$ by the subspace generated by elements of the form
\begin{equation}
 \label{ns}
  \xi b\otimes \eta - \xi \otimes \psi(b) \eta, 
  \end{equation}
 for all $\xi \in X,\ \eta \in Y,\ b \in B$.
 
 This has a natural structure of right module over $C$ given by
\[(\xi \otimes \eta)c = \xi \otimes (\eta c),\]
and a $C$-valued inner product defined on simple tensors as
\begin{equation}\label{eq:inn}
\inn{}{\xi_1 \otimes \eta_1 , \xi_2 \otimes \eta_2}{C} := \inn{}{\eta_1, {\psi(\inn{}{\xi_1,\xi_2}{B}) \eta }}{C} ,
\end{equation}
and extended by linearity.

The inner product is well-defined (cf. \cite[Proposition 4.5]{La95}); in particular, the null space 
$N=\{\zeta \in X\otimes_{\mathrm{alg}} Y \, ;\,  \inn{}{\zeta, \eta}{} = 0\}$ 
can be shown to coincide with the subspace generated by elements of the form in \eqref{ns}.

One then defines $X \hat{\otimes}_\psi Y$ to be the right Hilbert module obtained by completing $X \otimes_B Y $ in the norm induced by \eqref{eq:inn}.

Moreover for every $T \in \End^*_B(X)$, the operator defined on simple tensors by 
\[ \xi \otimes \eta \mapsto T(\xi) \otimes \eta
\]
extends to a well-defined operator $\phi_{*}(T) := T \otimes \id$. It is adjointable with adjoint given by $T^* \otimes \id = \phi_{*}(T^*)$. In particular, this means that there is a left action of $A$ defined on simple tensors by 
\[(\phi \otimes_{\psi} \id)(a)(\xi \otimes \eta) = \phi(a)\xi \otimes \eta,\]
and extended by linearity to a map \[
 \phi \otimes_{\psi} \id : A \rightarrow  \End^*_C(X\hot_{\psi}Y),\]
 thus turning $X \hot_{\psi} Y$ into an $(A,C)$-correspondence. For all the details, we refer the reader once more to \cite[Chapter 4]{La95}.

We remark that the interior tensor product induces an associative operation on isomorphism classes of \Cs correspondences.

\subsubsection{Kasparov modules and the theory of extensions}
We now come to defining the key objects in Kasparov's bivariant $K$-theory \cite{Kas80Ext}, which are inspired by the geometry of elliptic operators on manifolds.
\begin{defn}
An \emph{odd} Kasparov $(A,B)$-bimodule is a pair $(Y,F)$ where $Y=Y_{\phi}$ is a Hilbert $C^{*}$-correspondence from $A$ to $B$, and 
$F\in\End^{*}_{B}(Y)$ is a self-adjoint operator such that $F^{2}=1$ and $[F,a]\in\mathcal{K}(Y)$. An \emph{even} Kasparov module is a triple $(Y,F,\gamma)$ 
such that $(Y,F)$ is an odd Kasparov module and $\gamma\in\End^{*}_{B}(Y)$ is a self-adjoint unitary that commutes with $A$ and anticommutes with $F$.
\end{defn}
The natural equivalence relation of homotopy of Kasparov modules is conveniently defined via Kasparov modules for $(A,C([0,1],B))$. The homotopy classes
of odd Kasparov $(A,B)$-modules form an Abelian group denoted $KK_{1}(A,B)$. Similarly, the homotopy classes of even Kasparov modules form 
an Abelian group $KK_{0}(A,B)$. If we choose $A=\mathbb{C}$ then there are natural isomorphisms $KK_{*}(\mathbb{C},B)\simeq K_{*}(B)$, and as such
$KK$-theory generalises $K$-theory. The main feature of the theory is the existence of an associative, bilinear product structure
\begin{equation}
\label{eq: Kasprod}
KK_{i}(A,B)\times KK_{j}(B,C)\to KK_{i+j}(A,C),
\end{equation}
the \emph{Kasparov product}. Again, if we set $A=\mathbb{C}$, we see that elements in $KK_{j}(B,C)$ induce maps 
$K_{*}(B)\to K_{*+j}(C)$ by taking products from the right.

There is a close relationship between the Abelian groups $KK_{1}(A,B)$ and \linebreak $\Ext^{1}(A,B)$
which can be understood via the following \emph{Kasparov--Stinespring theorem}, first proved in \cite{Kas80St}.
\begin{thm}[see the proof of Theorem 3.2.7 in \cite{JenTho}]
\label{Stinespring}
Let $A,B$ be $C^{*}$-algebras, with $A$ separable and $B$ $\sigma$-unital. Let $X$ be a countably generated Hilbert $C^{*}$-module over $B$ and $\rho:A\to \End^{*}_{B}(X)$ be a  completely positive contraction. There exists a countably generated Hilbert $C^{*}$-module $Y$ over $B$, a $*$-homomorphism $\pi:A\to \End^{*}_{B}(Y)$ and an isometry $v:X\to Y$ such that $\rho(a)=v^{*}\pi(a)v$.
\end{thm}
It is worth noting that such an isometry $v:X\to Y$ immediately gives rise to a Toeplitz type algebra 
\[\mathcal{T}_{v}:=vv^{*}\End_{B}^{*}(Y)vv^{*}\simeq \End_{B}^{*}(X).  \]
To an extension with
a completely positive linear splitting $\ell:A\to E$ we can associate an odd Kasparov module by observing that, as $\mathcal{K}(X)$ is an ideal in $E$, there is a $*$-homomorphism $\varphi:E\to \End^{*}_{B}(X)$. We consider the completely positive contraction $\rho:=\varphi\circ\ell:A\to \End^{*}_{B}(X)$ and  obtain an $(A,B)$-bimodule $Y$ and an isometry $v:X\to Y$ via Theorem \ref{Stinespring}.
\begin{thm}
Let $X$ be a countably generated Hilbert $C^{*}$-module over the $\sigma$-unital $C^{*}$-algebra $B$ and $A$ a separable $C^{*}$-algebra. If
\[\xymatrix{0 \ar[r] & \mathcal{K}(X) \ar[r] & E \ar[r]^{} & A\ar[r] & 0,}\]
is a semisplit extension with completely contractive and completely positive linear splitting $\ell:A\to E$, then the Stinespring dilation $v:X\to Y$ of $\rho:=\varphi\circ\ell:A\to \End^{*}_{B}(X)$ makes $(Y, 2vv^{*}-1)$ into an 
odd Kasparov module for $(A,B)$.
\end{thm}
\begin{proof}
As $Y$ is an $(A,B)$-correspondence and $F=2vv^{*}-1$ it holds that $F^{2}=1$ and $F^{*}=F$. 
Hence all we need to check is that $[F,\pi(a)]=2[vv^{*},\pi(a)]$ is an element of $\mathcal{K}(Y)$. Write $p=vv^{*}$, so $p^{2}=p^{*}=p$ and 
\[[p,\pi(a)]=p\pi(a)(1-p)-(1-p)\pi(a)p.\]
It thus suffices to show that $p\pi(a)(1-p)\pi(a)^{*}p\in \mathcal{K}(Y)$, for $\mathcal{K}(Y)$ is an ideal in $\End^{*}_{B}(Y)$ and thus
for $T\in\End^{*}_{B}(Y)$ it holds that $T\in\mathcal{K}(Y)$ if and only if $TT^{*}\in\mathcal{K}(Y)$ (see for instance  \cite[Proposition II.5.1.1.ii]{Bl06}). Now
$v\mathcal{K}(X)v^{*}\subset \mathcal{K}(Y)$ since for $x_{1},x_{2}\in X$ it holds that $v \theta_{x_{1}, x_{2}}v^{*}= \theta_{v(x_{1}),v(x_{2})}$, and we compute
\begin{align*}
p\pi(a)(1-p)\pi(a^{*})p &=vv^{*}\pi(a)(1-vv^{*})\pi(a^{*})vv^{*}\\
&=v(v^{*}\pi(a)vv^{*}\pi(a^{*})v-v^{*}\pi(aa^{*})v)v^{*}\\
&=v(\ell(a)\ell(a^{*})-\ell(aa^{*}))v^{*}\in v\mathcal{K}(X)v^{*}.
\end{align*}
This proves that $(Y,F)$ is a Kasparov module.
\end{proof}
By the previous theorem, we see that an extension of $C^{*}$-algebras induces an element in $KK_{1}(A,B)$. Using the product structure \eqref{eq: Kasprod}, 
this leads to the elegant viewpoint that an extension induces maps
\[\otimes_{A}[(Y,F)]: K_{*}(A)\to K_{*+1}(B),\]
via the Kasparov product. These maps coincide with the boundary maps in the long exact sequence associated to the extension. For instance, the product
with the extension
\[\xymatrix{0 \ar[r] & \mathcal{K}\overline{\otimes} A \ar[r] & \mathcal{T}_{0}\overline{\otimes} A \ar[r] & C_{0}(0,1)\overline{\otimes}A \ar[r] & 0},\]
of the previous section induces the Bott periodicity isomorphisms $K_{n}(S^{2}A)\simeq K_{n}(A)$. In fact, the extension above, in combination with the Kasparov product,
can be used to prove the general bivariant Bott periodicity isomorphisms
\[KK_{*}(S^{2}A,B)\simeq KK_{*}(A,B)\simeq KK_{*}(A,S^{2}B), \]
for any pair of separable $C^{*}$-algebras $(A,B)$.

The Kasparov--Stinespring construction can be inverted up to homotopy, yielding the statement that $KK_{1}(A,B)$ is isomorphic to $\Ext^{1}(A,B)$. 
Effectively, this amounts to the observation that $KK$-theory is nothing but the study of extensions of $C^{*}$-algebras.

To conclude, let us sketch the inverse construction. An odd Kasparov module $(X,F)$ for $(A,B)$ defines an adjointable projection $P:=\frac{1}{2}(F+1)$ and hence a complemented submodule $X:=PY\subset Y$. The $C^{*}$-subalgebra
\[E:=\left\{(PTP,a)\in \End^{*}_{B}(X)\oplus A: T\in\End^{*}_{B}(Y),\quad P(T-a)P\in\mathcal{K}(Y)\right\},\]
of $\End^{*}_{B}(Y)\oplus A$
is an extension of $A$ by $\mathcal{K}(X)$. 
To see that $E$ is closed under products, we use that
\begin{align*}
PSPTP-PabP&=P(S-a)PTP+PaP(T-b)P-Pa(1-P)bP\\
&=P(S-a)PTP+PaP(T-b)P-[P,a](1-P)bP,
\end{align*}
which is an element of $\mathcal{K}(X)$. It admits the completely contractive linear splitting 
\[\ell: A\to E,\quad \ell:a\mapsto (PaP,a),\] 
and the inclusion $\mathcal{K}(X)=\mathcal{K}(PY)\to E$ defined by $T\mapsto (T,0)$ and the quotient map $(PTP,a)\to a$ with kernel $\mathcal{K}(X)$.
The $C^{*}$-algebra $E$ can be viewed as an \emph{abstract Toeplitz algebra} associated to the Kasparov module $(Y,F)$. This inverts the Kasparov--Stinespring construction, as is easily checked.
\color{black}

\section{Toeplitz algebras, crossed products by the integers, and Cuntz--Pimsner algebras}
\label{sec:ToePim}
We will now describe two constructions of Toeplitz \Cs algebras and quotients thereof that appear in the study of solid state systems, as they provide the natural framework for implementing the bulk-edge correspondence.
\subsection{Crossed products by the integers and the Pimsner--Voiculescu Toe\-plitz algebra}
Our first object of study are crossed products by the integers. They constitute one of the simplest and most well-understood examples of \Cs dynamical systems, a class of objects which were introduced to study group actions on \Cs algebras.

Let $B$ be a unital \Cs algebra, and $\alpha \in \Aut(B)$ a single automorphism. This defines an action of the additive group $\mathbb{Z}$ of integers on $B$ given by
\[\Z \to \Aut(B), \quad n \mapsto \alpha^n.
\] The crossed product \Cs algebra $B \rtimes_{\alpha}\Z$ is realised as the universal \Cs algebra generated by $B$ and a unitary $u$ satisfying the covariance condition 
\[ \alpha^n(b)=  u^n b u^{*n}, \quad \forall b \in B, n \in \Z.\]
As described in \cite{PiVo80}, crossed products by a single automorphism can be realised as quotients in a Toeplitz exact sequence of \Cs algebras, constructed starting from the Toeplitz extension \eqref{eq:Toe_ext}.

\begin{defn}
Let $B$ a unital \Cs algebra and $\alpha$ an automorphism and let $\mathcal{T}= C^*(T)$ be the Toeplitz algbera of the unilateral shift. The \emph{Pimsner--Voiculescu Toeplitz algebra} $\mathcal{T}(B,\alpha)$ is defined as the \Cs subalgebra of $B \overline{\otimes} \mathcal{T}$ generated by $B \otimes 1$ and $u \otimes T$.
\end{defn}

The Pimsner--Voiculescu Toeplitz algebra $\mathcal{T}(B,\alpha)$ and the crossed product \Cs algebra $B\rtimes_{\alpha}\Z$ fit into a short exact sequence involving the stabilisation of $B$:
 \beq
 \label{eq:ToeCProd}
\xymatrix{ 0 \ar[r] &  \mathcal{K}\overline{\otimes} B\ar[r] & \mathcal{T}(B,\alpha) \ar[r] & B \rtimes_{\alpha} \Z \ar[r] &0.
}
 \eeq
 Proof of exactness of the above sequence follows after tensoring the Toeplitz exact sequence \eqref{eq:Toe_ext} with the algebra $B$, using nuclearity of $C(S^1)$ together with Lemma~\ref{exactextension}, and by realising $B\rtimes_{\alpha}\Z$ as a subalgebra of $B \overline{\otimes} C(S^1)$ (see \cite[Section 2]{PiVo80}).
 
The Pimsner--Voiculescu Toeplitz algebra $\mathcal{T}(B,\alpha)$ is $KK$-equivalent to the algebra $B$ itself. The exact sequence \eqref{eq:ToeCProd} then induces six-term exact sequences that allow for an elegant computation of the $K$-theory and $K$-homology groups of the crossed product algebra $B \rtimes_{\alpha} \Z$ in terms of those of the algebra $B$. These exact sequences are a special case of those described in Subsection \ref{sec: exseq}.

\subsection{Pimsner's construction: universal \Cs algebras from \Cs correspon\- dences}
The construction which we shall describe now generalises that of crossed products by the integers. In \cite{Pi97}, starting from a \Cs correspondence $(X,\phi)$, Pimsner constructed two \Cs algebras $\mathcal{T}_X$ and $\mathcal{O}_X$, which are now referred to as the \emph{Toeplitz algebra} and the \emph{Cuntz--Pimsner algebra} of the pair $(X,\phi)$, respectively. Both algebras are characterized by universal properties and depend only on the isomorphism class of the pair $(X,\phi)$. We will describe the construction for compact correspondences.

\subsubsection{The Toeplitz algebra}
\label{ss:PimBimod}
As one can take balanced tensor products of \Cs corre\- spondences, as described in \ref{sss:hilb}, we consider the modules
\beq
\label{eq:Epowerk}
 X^{(k)} := X^{{\hot}^k_{\phi}} \quad k>0,
 \eeq and we take the infinite direct sum
\begin{equation}
 F_X= B \oplus \bigoplus_{k=1}^{\infty} X^{(k)},
\end{equation}
which is referred to as the \emph{(positive) Fock correspondence} associated to the correspondence $(X, \phi)$.

One can naturally associate to any element $\xi \in X$ a shift map:
\begin{equation}
 T_\xi (\xi_1 \ot \dots \ot \xi_k) = \xi \ot \xi_1 \ot \dots \ot \xi_k, \qquad T_\xi (b)= \xi b.
\end{equation}
This is an adjointable operator on $F_X$, with adjoint
\begin{equation}
 T^*_\xi (\xi_1 \ot \dots \ot \xi_k) = \phi( \langle \xi, \xi_1 \rangle )  \xi_2 \ot \dots \ot \xi_k, \qquad T^{*}_{\xi}(b)=0.
\end{equation}
\begin{defn}
The Toeplitz algebra of the \Cs correspondence $X_{\phi}$ is the smallest \Cs subalgebra of $\BndB(F_{X})$ that contains all the $T_\xi$ for $\xi \in X$. 
\end{defn}
When $(X,\phi)$ is a compact $C^{*}$-correspondence, the compact operators on the Fock module sit inside $\mathcal{T}_E$ as a two-sided ideal, motivating the following:
\begin{defn}
The Cuntz--Pimsner algebra $\Pim$ of a compact \Cs correspondence $(X, \phi)$ is the quotient algebra appearing in the exact sequence
\begin{equation}
\label{eq:PimES}
 \xymatrix{0 \ar[r]& \mathcal{K}_B(F_X) \ar[r] & \mathcal{T}_{X} \ar[r]^{\pi} & \mathcal{O}_{X} \ar[r]& 0.}
\end{equation}
The image of an element $T_{\xi} \in \mathcal{T}_{X}$ under the quotient map $\pi$ will be denoted by $S_{\xi}$.
  \end{defn}  
  
Changing the ideal in the exact sequence \eqref{eq:PimES}, one can define the Cuntz--Pimsner algebra of a general (i.e. non-compact, and possibly non-injective) $C^{*}$-correspon\-dence. We will not be concerned with this more elaborate construction here. For details see \cite{Pi97, Kat04}

Many well-known and studied examples of \Cs algebras admit a description as Toeplitz--Pimsner and Cuntz--Pimsner algebras. The theory provides a unifying framework for a variety of examples, ranging from the study of discrete dynamics to more geometric situations.
 \begin{ex}
 \label{ex:Cuntz}
Let $B=\mathbb{C}$ and $X=\C^n$ and $\phi$ the left action by multiplication. If one chooses a basis for $\C^n$, then the Toeplitz algebra of $(X, \phi)$ is the universal \Cs algebras generated by $n$ isometries $V_1, \dots, V_n$ satisfying $\sum_{i} V_i V_i^* \leq 1$.

This yields the well known Toeplitz extension for the Cuntz algebras~$\mathcal{O}_n$:
  \[
\xymatrix{ 0 \ar[r] & \mathcal{K}(\mathcal{F}) \ar[r] & C^*(V_1, \dots, V_n) \ar[r] & \mathcal{O}_n \ar[r] &0,
} \]
where $\mathcal{F}$ is the full Fock space on $\C^n$. In particular, for $n=1$ one gets back the classical Toeplitz extension of \eqref{eq:Toe_ext}.
 \end{ex}

\begin{ex}[{cf. \cite[Section 2]{KPW98})}]
If the correspondence $X$ is a finitely generated and projective module over a unital \Cs algebras, the Pimsner algebra of $(X, \phi)$ can be realized explicitly in terms of generators and relations.
Indeed, since $X$ is finitely generated and projective, there exists a finite set $\lbrace \eta_j \rbrace_{j=1}^{n}$ of elements of $X$ such that 
\[\xi = \sum\nolimits_{j=1}^n  \eta_j \inn{}{\eta_j, \xi}{B}, \qquad \forall \xi \in X.
\]
Then, using the above formula, one can spell out the left $B$-action on $X$ as
$$
\phi(b) \eta_j= \sum_{j=1}^n \eta_i \inn{}{\eta_i, \phi(b) \eta_j}{B}, \quad \forall b \in B.
$$
The \mbox{\Cs algebra} $\Pim$ is then the universal \mbox{\Cs algebra} generated by $B$ together with 
$n$ operators $S_1, \dots, S_n$, satisfying
\begin{align}
& S_i^* S_j = \inn{}{\eta_i, \eta_j}{B}, \quad \sum\nolimits_j S_j S_j ^* = 1, \quad \mbox{and} \quad 
b S_j = \sum\nolimits_{i}S_i \inn{}{\eta_i, \phi(b) \eta_j}{B},
\end{align}
for $b \in B$, and $j=1,\dots, n$. The generators $S_i$ are partial isometries if and only if $\inn{}{\eta_i,\eta_j}{}=0$ for $i\neq j$. For $B=\C$ and $E$ a Hilbert space of dimension $n$, one recovers the Cuntz algebra $\mathcal{O}_n$ of Example \ref{ex:Cuntz}.
\end{ex}

\begin{ex}
\label{ex:CPZ}
Let $B$ be a \Cs algebra and $\alpha: B \rightarrow B$ an automorphism of $B$. Then $X=B$, seen as a module over itself, can be naturally made into a compact \Cs correspondence. 

The right Hilbert $B$-module structure is the standard one, with right $B$-valued inner product $\inn{}{a,b}{B} = a^* b$. The automorphism $\alpha $ is used to define the left action via $a \cdot b = \alpha(a) b $ and left $B$-valued inner product given by $\inn{B}{a,b}{} = \alpha(a^*b)$. 

Each module $X^{(k)} $ is isomorphic to $B$ as a right-module, with left action 
\begin{equation}
\label{autaction}
a \cdot (x_1 \ot \cdots \ot x_k) = \alpha^k(a) \alpha^{k-1}(x_1) \cdots \alpha(x_{k-1}) x_k.
\end{equation} 
The corresponding Pimsner algebra $\Pim$ coincides then with the crossed product algebra $B \rtimes_{\alpha} \Z$, while the Toeplitz algebra $\mathcal{T}_{X}$ agrees with the Toeplitz algebra $\mathcal{T}(B,\alpha)$. The extension \eqref{eq:PimES} then reduces to \eqref{eq:ToeCProd}.
\end{ex}

\subsubsection{Six-term exact sequences}
\label{sec: exseq}
As the Toeplitz extension \eqref{eq:PimES} is semi-split whenever the coefficient algebra $B$ is nuclear, it induces six-term exact sequences in $KK$-theory. These exact sequences can be simplified to a great extent after making the following observations:
\begin{itemize}
\item For a compact $C^{*}$-correspondence $(X,\phi)$, the triple$(X, \phi, 0)$ gives a well-defined even Kasparov module (with trivial grading), whose class we denote by $[X]$.
\item The ideal $\mathcal{K}(F_X)$ is naturally Morita equivalent to the algebra $B$ itself.
\item By \cite[Theorem 4.4.]{Pi97}, the Toeplitz algebra $\mathcal{T}_X$ is $KK$-equivalent to the coefficient algebra $B$.
\end{itemize}
In $K$-theory, the induced six-term exact sequence reads
\begin{equation}
\label{eq:6tesPim}
\xymatrix{
&K_0(B) \ar[r]^{\otimes(1 - [X])} & K_0(B) \ar[r]^{i_*} &K_0(\Pim) \ar[d]^{\partial} \\
&K_1(\Pim) \ar[u]^{\partial} &K_1(B)\ar[l]^{i_*} & K_1(B) \ar[l]^{\otimes(1 - [X])}},
\end{equation}
where $i_{*}$ is the map induced by the inclusion $B \hookrightarrow \Pim$ and the maps $\partial$ are connecting homomorphisms. Up to Morita equivalence, the latter can be computed as Kasparov products with the class of the extension \eqref{eq:PimES}. An unbounded representative for the extension class was constructed \cite{GMR18} in the setting bi-Hilbertian bimodules of finite Jones--Watatani index (cf. \cite{KPW04}), subject to some additional assumption.

We conclude this section by remarking that, in the case of a self-Morita equivalence bimodule---i.e., whenever $X$ is full and $\phi$ implements an isomorphism between $B$ and $\K_B(X)$---the exact sequence \eqref{eq:6tesPim} can be interpreted as a generalization of the classical \emph{Gysin sequence} in $K$-theory (see \cite[IV.1.13]{Ka78}) 
for the module of sections $E$ of a noncommutative line bundle. The Kasparov product with the map $1 - [X]$ can be interpreted as a \emph{noncommutative Euler class}. This analogy was exploited in \cite{AKL} to compute $K$-theory groups of algebras presenting a circle bundle structure.

\section{Applications to topological insulators}
\label{sec:SolidState}
We conclude by discussing the \emph{bulk-edge correspondence}, a principle in solid state physics, according to which one should be able to \emph{read} the topology of the bulk physical system from the effects it induces on boundary states. This principle underlies, for example, the quantization of the Hall current on the boundary of a sample of a quantum Hall system. 

In this section, we illustrate how Toeplitz extensions and the maps they induce in (bivariant) $K$-theory are essential for a mathematical understanding of these phenomena.

\subsection{The bulk-boundary correspondence for the one-dimensional {Su-- Sch\-rieffer--Heeger} model and the Noether--Gohberg--Krein index theorem}
We will now give an exposition of the key ideas behind the bulk-edge correspondence for the one-dimensional Su--Schrieffer--Heeger model \cite{SuSchHee}, a lattice model with chiral symmetry. Our main reference for this Subsection is \cite[Chapter 1]{ProSB}. 
On the Hilbert space $\C^2 \otimes \C^ n \otimes \ell^2(\Z)$ we consider the one dimensional Hamiltonian
\begin{equation}
\label{eq:SSHHam}
H:=\frac{1}{2}(\sigma_1+ i \sigma_2)\otimes \id_n \otimes U + \frac{1}{2}(\sigma_1- i \sigma_2)\otimes \id_n \otimes U^* +m \sigma_2 \otimes \id_n \otimes \id, 	
\end{equation}
where $\id_n$ and $\id$ are identity operators on $\C^n$ and $\C^2$, respectively, $m$ is a mass term, $U$ is the right shift on $\ell^2(\Z)$ defined in \eqref{eq:rightUshift}, and the $\sigma_{i}$ are the Pauli matrices 
\[\sigma_1=\begin{pmatrix} 0 & 1\\ 1&0\end{pmatrix}, \qquad \sigma_2=\begin{pmatrix} 0 &- i\\i & 0 \end{pmatrix}, \qquad \sigma_3 = \begin{pmatrix} 1 &0\\0 &-1 \end{pmatrix}.\]
This Hamiltonian goes back to work of \cite{SuSchHee} and models a conducting polymer, namely polyacetilene. It possess a chiral symmetry, implemented by the unitary operator
\[J= \sigma_3 \otimes \id_n \otimes \id,\]
i.e., $J^ *HJ=-H$.

The model has a spectral gap at $m=0$ so there exists $\varepsilon>0$ and a continuous function 
\[\chi:\mathbb{R}\to\mathbb{R},\quad \chi(x)=\left\{\begin{matrix} 0 &\textnormal{for } x\in &(-\infty, -\varepsilon]\\ 1 & \textnormal{for } x\in& [0,\infty), \end{matrix}\right.\]
so that we can form the \emph{Fermi projection} 
$P_F := \chi(H)$ through functional calculus with $\chi$.
 The projection $P_{F}$ satisfies the identity $JP_FJ=1-P_F$, so that the flat band Hamiltonian 
\[Q:=\id - 2P_F = \mathrm{sgn}(H)\]
satisfies again $J^*QJ=-Q$. Moreover, $Q^2=1$, hence its spectrum consists of the two isolated points $+1$ and $-1$, allowing us to write 
\[Q= \begin{pmatrix} 0 & U_F^*\\
 U_F & 0	
 \end{pmatrix}
 \]
for $U_F$ a unitary on $\C^n \otimes \ell^2(\Z).$ This unitary operator, called the \emph{Fermi unitary}, provides us with a natural topological invariant for the boundary system, the first odd Chern number, which can be computed as follows.

We use the discrete Fourier transform mentioned in \eqref{eq:Fourier} to write 
$\mathcal{F}Q\mathcal{F}^*$ as a direct integral $\int_{S^1}^{\oplus} Q_z \mathrm{d}_z$ where each of the $Q_k$'s has the form
\[Q_z=\begin{pmatrix}
	0 & U_z^* \\
	U_z & 0
\end{pmatrix}.\]
The family of unitary operators is differentiable and the first Chern class can be computed as the integral
\begin{equation}
\label{eq:FermiWN}
\mathrm{Ch}_1 (U_F) := \frac{i}{2\pi}\int_{S^1}^{\oplus} \tr(U_z \partial_z U_z) \mathrm{d}z  	
\end{equation}
This quantity is an invariant under \emph{small perturbations}.

\subsubsection{The bulk boundary correspondence}
We now introduce an edge for the Hamiltonian \eqref{eq:SSHHam} by restricting it to the Hilbert space $\C^2 \otimes \C^n \otimes \ell^2(\N)$ and imposing Dirichlet boundary conditions. The resulting Hamiltonian is 
\begin{equation}
\label{eq:SSHedgeHam}
\widehat{H}:=\frac{1}{2}(\sigma_1+ i \sigma_2)\otimes \id_n \otimes T + \frac{1}{2}(\sigma_1- i \sigma_2)\otimes \id_n \otimes T^* +m \sigma_2 \otimes \id_n \otimes \id, 	
\end{equation}
with conventions as above, and with $S$ the unilateral shift on $\ell^2(\N)$ described in Subsection~\ref{ss:shifts}.
Similarly to the bulk Hamiltonian, the edge Hamiltonian has a chiral symmetry implemented by the half-space chiral operator $\widehat{J}=\sigma_3 \otimes \id_n \otimes \id$. Moreover, it has a spectral gap at 0 that we denote by $\Delta$.

Let us now consider the Hilbert space obtained as the span of all the eigenvectors with eigenvalues in $[-\delta, \delta] \subset \Delta$, which we denote by $\mathcal{E}^{\delta}$. The chirality operator $\widehat{J}$ can be diagonalised on $\mathcal{E}^{\delta}$, and we have a splitting $\mathcal{E}^{\delta} = \mathcal{E}^\delta_{+} \oplus \mathcal{E}^{\delta}_{-}$.

The difference of the dimensions of the spaces $\mathcal{E}^{\delta}_{\pm}$ is the \emph{boundary invariant} of the system and it can be computed as a trace:
\[ \tr(\widehat{J} \hat{P}_{\delta}) = N_{+} - N_{-}, \qquad N_{\pm}= \dim \mathcal{E}^{\delta}_{\pm},\]
where $\hat{P}_{\delta} := \chi(\vert \hat{H} \vert \leq \delta)$ is the spectral projection. This invariant is independend of the choice of $\delta$, as long as it lies in the central gap.

The bulk-edge correspondence is contained in the following identity, that relates the bulk invariant (winding number of the Fermi unitary) to the boundary invariant we just introduced.
\begin{thm} [\!\!{\cite[Theorem 1.2.2]{ProSB}}]
Consider the Hamiltonian \eqref{eq:SSHHam} and its half-space restriction \eqref{eq:SSHedgeHam}. If $U_F$ is the Fermi unitary and $\mathrm{Ch}_1(U_F)$ its winding number defined in \eqref{eq:FermiWN}, then
\[  \mathrm{Ch}_1(U_F) = \mathrm{Tr}(\tilde{J}\tilde{P}(\delta)). \]
\end{thm}

We remark that the Toeplitz extension \eqref{eq:Toe_ext} offers an index theoretic interpretation of this identity. The above equality of classes follows from the six-term exact sequence coming from the Toeplitz extension \eqref{eq:Toe_ext}. Indeed, the boundary map described in \eqref{eq:ToeBndry} maps classes of unitaries from the bulk algebra $C(S^1)$ to classes of projections in the
boundary algebra $\mathcal{K}(\ell^2(\N))$, whose $K$-theory classes are given by the winding number of the relevant unitary.
 
\subsection{The role of Toeplitz extensions in the bulk-edge correspondence}
The example of the Su--Schrieffer--Heeger model is in some sense paradigmatic, as other solid state systems can be modelled using related \Cs algebraic extensions, where Toeplitz algebras serve as models for the half-space system, while quotients of Toeplitz algebras are used to model the edge system. Likewise, the $K$-theory boundary map coming from the extension can be used to implement the bulk-edge correspondence, relating bulk invariants to edge invariants. We will now recall some recent results.

In \cite{BoKeRe}, the observable
algebra of the physical system is a crossed product \Cs algebra. The Toeplitz extensions for (twisted) crossed products by $\mathbb{Z}^n$
\[\xymatrix{
0 \ar[r]& C\otimes \mathcal{K}(\ell^2(\mathbb{N}))\ar[r] & \mathcal{T}(\beta) \ar[r] & C \rtimes_\beta \mathbb{Z}^n \ar[r] & 0
},\]
offers the natural framework for the investigation of the bulk-edge correspondence, as it elegantly links the algebras of the bulk and the edge systems.

In \cite{BM19}, the authors replace crossed products \Cs algebras by groupoid \Cs algebras. While crossed products are naturally an example of groupoid \Cs algebras, the advantage of this more general setting lies in the possibility of studying systems without translational symmetries, like those resulting from non-periodic $\mathbb{R}^d$-actions. The systems are still linked by a short exact sequence of the form
\[\xymatrix{
0 \ar[r] & C^*_r(\mathcal{Y},\sigma) \otimes \mathcal{K} \ar[r] & \mathcal{T} \ar[r] &  C^*_r(\mathcal{G},\sigma) \ar[r] & 0},\]
where $\sigma$ is a $2$-cocycle encoding the magnetic field, $\mathcal{Y}$ is a closed subgroupoid of the groupoid $\mathcal{G}$, and the algebra $\mathcal{T}$ models the half-space system.

Quite remarkably, in the one-dimensional case, the groupoid \Cs algebra admits an alternative description as Cuntz--Pimsner algebra of a self-Morita equivalence bimodule (cf. \cite[Subsection 2.3]{BM19}) . The map implementing the bulk-edge correspondence is realised as a Kasparov product with the unbounded representative for the class of the extension \eqref{eq:PimES}, as constructed in  \cite{GMR18} (see also \cite{AKL}). It remains an interesting open question whether groupoid $C^{*}$-algebras of higher dimensional systems admit a description in terms of \Cs algebras associated to families of \Cs correspondences, for instance in terms of product and subproduct systems \cite{Fow02,Fle18,ShSo09,Vis12}. 
\subsection*{Acknowledgment}
We are indebted to our colleagues and collaborators C. Bourne, M. Goffeng, and J. Kaad for inspiring conversations on topics related to this manuscript.

\bibliographystyle{abbrv}
\bibliography{AM_Toeplitz_arXiv.bib}

\end{document}